\documentclass[11pt]{amsart}
\usepackage{geometry}                
\usepackage{url}
\usepackage{graphicx}
\usepackage{amssymb}
\usepackage{epstopdf}
\usepackage{color}
\usepackage{amsmath}
\usepackage{url}
\usepackage[square,numbers]{natbib}
\usepackage{subeqnarray}
\definecolor{darkred}{rgb}{.6,0,0}
\definecolor{darkblue}{rgb}{0,0,.7}

\newtheorem{lemma}{Lemma}[section]
\newtheorem{example}{Example}[section]
\newtheorem{remark}{Remark}[section]

\title[FBSDEs for minimum variance estimation]{On forward-backward SDE approaches to continuous-time minimum variance 
estimation}
\author{Jin Won Kim \and Sebastian Reich}
\address{Institut f\"ur Mathematik, Universit\"at Potsdam, Karl-Liebknecht-Str. 24/25, 14476 Potsdam}
\date{\today}                                           

\begin{document}
\maketitle

\begin{abstract} The work of Kalman and Bucy has established a duality between filtering and optimal estimation in the context of time-continuous linear systems. This duality has recently been extended to time-continuous nonlinear systems in terms of an optimization problem constrained by a backward stochastic partial differential equation. Here we revisit this problem from the perspective of appropriate forward-backward stochastic differential equations. This approach sheds new light on the estimation problem and provides a unifying perspective. It is also demonstrated that certain formulations of the estimation problem lead to deterministic formulations similar to the linear Gaussian case as originally investigated by Kalman and Bucy. Finally, optimal control of partially observed diffusion processes is discussed as an application of the proposed estimators.
\end{abstract}
%
\section{Introduction}
%

In this paper, we revisit the problem of estimation in the context of continuous time filtering problems \cite{jazwinski2007stochastic,bain2008fundamentals}. 
Our interest has been triggered by recent progress on the topic \cite{Kim22,KimMehta2022a,KimMehta2022b}, which has extending the dual optimal control perspective for linear problems, as originally exposed by \cite{kalman1960new,kalman1961new,jazwinski2007stochastic}, 
to nonlinear filtering problems. The key idea is to consider an estimator in the form of
\begin{equation} \label{eq:est}
    \mathcal{S}_T[f] = \mu[{\rm Y}_0] - \int_0^T \mathcal{U}_t^{\rm T}{\rm d}Z_t,
\end{equation}
where $f(x)$ is the function which the conditional expectation value is to be estimated of. Furthermore, $\mu$ is the density of states at initial time $t=0$ and $Z_{0:T}$ is the observation path, which are both given. The function ${\rm Y}_0(x)$ and the control $\mathcal{U}_{0:T}$ are to be determined such that $\pi_T[f] = \mathcal{S}_T[f]$. Here $\pi_T$ denotes the filtering distribution at final time $t=T$.  As proposed in the work of \cite{Kim22,KimMehta2022a,KimMehta2022b}, the function ${\rm Y}_t(x)$, and the control $\mathcal{U}_t$, $t\in [0,T]$, satisfy the backward stochastic partial differential equation (BSPDE)
\begin{equation} \label{eq:BSPDE_K}
    -{\rm d}{\rm Y}_t = \mathcal{L}^x \,{\rm Y}_t {\rm d}t- {\rm V}_t^{\rm T}{\rm d}Z_t
    + (\mathcal{U}_t+{\rm V}_t)^{\rm T}h \,{\rm d}t, \quad {\rm Y}_T = f.
\end{equation}
The minimum variance property of the estimator is then expressed in terms of an appropriate cost function which determined the optimal $\mathcal{U}_{0:T}$ uniquely. Here $\mathcal{L}^x$ denotes the generator of the underlying signal process and $h(x)$ is the forward map in the observation process.

We address the same estimation problem from the perspective of forward-backward stochastic differential equations (FBSDEs) instead \cite{Carmona}. Our approach allows for a unified perspective on the following four scenarios:
\begin{itemize}
    \item[(i)] estimating $\pi_T[f]$ using an estimator of the form (\ref{eq:est}),
    \item[(ii)] estimating $\sigma_T[f]$ using an estimator of the form (\ref{eq:est}), where $\sigma_t$ denotes the non-normalized filtering distribution satisfying the Zakai equation \cite{jazwinski2007stochastic,bain2008fundamentals},
    \item[(iii)] estimating $\pi_T[f]$ using the innovation 
\begin{equation} \label{eq:innovation}
I_t = Z_t - \int_0^t \pi_s[h]{\rm d}s
\end{equation}
in the estimator (\ref{eq:est}) instead of the observations $Z_t$, 
\item[(iv)] estimating $\sigma_T[f]$ using the induced observation error
\begin{equation} \label{eq:obs_error}
W_t = Z_t - \int_0^t h(X_s){\rm d}s
\end{equation}
instead of the observations $Z_t$ directly.
\end{itemize}
Our FBSDE approach allows for a transparent definition of minimum variance estimation in all four cases. It is found that (ii)-(iv) lead to deterministic backward Kolmogorov equations \cite{Pavliotis2016,Carmona} in the unknown ${\rm Y}_t(x)$ instead of the BSPDE formulation (\ref{eq:BSPDE_K}), which arises from (i). Furthermore, while (ii) and (iii) lead to the standard backward Kolmogorov equation, (iv) introduces an additional drift term in the spirit of Feynman--Kac formulations \cite{Carmona}. By strictly following a FBSDE approach, a numerical treatment also becomes feasible following established computational methods for FBSDEs \cite{CKSY21}. This aspect will be explored in a future publication.

We emphasize that we assume throughout this paper that the observable $f(x)$ is a deterministic function independent of the data $Z_{0:T}$. On the contrary, applications of estimators of the form (\ref{eq:est}) to the problem of filter stability have been discussed in \cite{Kim22,KimMehta2023c}. In this context, the terminal condition in (\ref{eq:BSPDE_K}) becomes data dependent and the analysis of the resulting FBSDEs becomes more involved. Such an extension is left to future work. Another possible extension concerns estimators for more general conditional expectation values such as
\begin{equation} \label{eq:general_expectation}
\mathcal{V}_T := \int_0^T \pi_t[c_t]{\rm d}t + \pi_T[f]
\end{equation}
with given functions $c_t(x)$, $t\in [0,T]$, and $f(x)$. In this paper, we briefly discuss an application of (\ref{eq:general_expectation}) in the context of optimal control of partially observed diffusion processes  \cite{Bensoussan92,Borkar}.

The remainder of this paper is structured as follows. Section \ref{sec:filtering} provides the necessary background on time-continuous nonlinear filtering and introduces two sets of forward SDEs (FSDEs), which will be play an essential role in deriving the optimal estimators for scenarios (i)-(iv). The estimators for the conditional expectation values are discussed for all four scenarios separately in Section \ref{sec:estimation}. Emphasis is put on a unified framework using appropriate FBSDEs. It is demonstrated that the FBSDE become equivalent to a backward Kolmogorov equation under scenarios (ii) and (iii). The FBSDE perspective also leads to novel formulations of the stochastic optimal estimation problem (i) while scenario (iv) results in a backward PDE formulation in the spirit of Feynman--Kac. Section \ref{sec:SOC} provides an application of (\ref{eq:general_expectation}) to problems from optimal control of partially observed diffusion processes \cite{Bensoussan92,Borkar}.
Our paper ends with some conclusions. 

%
\section{The continuous time filtering problem} \label{sec:filtering}
%

The joint measure over paths $(X_{0:T},Z_{0:T})$ in state variable $x$ and observed variable $z$ is determined by the FSDE
\begin{subequations} \label{eq:FSDEs0}
\begin{align} \label{eq:state_n}
{\rm d}X_t &= b(X_t){\rm d}t + \sigma {\rm d}B_t,\qquad X_0 \sim \mu,\\ \label{eq:output_n}
{\rm d}Z_t &= h(X_t){\rm d}t + {\rm d}W_t, \qquad  \,\, Z_0 = 0,
\end{align}
\end{subequations}
for time $t \in [0,T]$. Here $(B_t,W_t)$ are independent Brownian motions. 
We denote the induced path measure by $\mathbb{P}$ and the expectation value of a function $g(x,z)$  by $\mathbb{E}[g(X_t,Z_t)]$. We are interested in the
conditional path measure $X_t|Z_{0:t}$ and denotes its density at time $t$, by $\pi_t$, $t \in [0,T]$ \cite{jazwinski2007stochastic,bain2008fundamentals}. Recall that the filtering distribution $\pi_t$ satisfies the Kushner--Stratonovitch (KS) equation of nonlinear filtering \cite{jazwinski2007stochastic,bain2008fundamentals}.

It is also well-known that the innovation process
(\ref{eq:innovation}) behaves like Brownian motion independent of $B_t$ under the path measure $\mathbb{P}$ and we introduce the FSDE
\begin{subequations} \label{eq:FSDEs1}
    \begin{align} \label{eq:FSDE1}
        {\rm d}X_t &=b(X_t){\rm d}t + \sigma {\rm d}B_t, \qquad  \qquad \qquad  X_0 \sim \mu,\\ \label{eq:FSDE2}
        {\rm d} D_t &= D_t \,(h(X_t)-\pi_t[h])^{\rm T}{\rm d}I_t, \quad 
        \qquad     \, D_0 = 1,
     \end{align}
\end{subequations}
which are driven by the independent Brownian processes $(B_t,I_t)$. We denote the path measure induced by the FSDE (\ref{eq:FSDEs1}) by $\mathbb{P}^\ast$ and its expectation operator by $\mathbb{E}^\ast$. It holds that
\begin{equation}
\pi_t[f] = \mathbb{E}^\ast [D_t f(X_t)|I_{0:t}].
\end{equation}
We also introduce the shorthand notation
\begin{equation}
    \mathbb{E}^\ast_I[g(X_t,D_t)] = \mathbb{E}^\ast [g(X_t,D_t)|I_{0:t}],
\end{equation}
for any suitable function $g(x,d)$.


The SDE (\ref{eq:FSDE2}) is of mean-field type, which becomes more transparent when rewritten in the form
\begin{equation}
{\rm d} D_t = D_t \,(h(X_t)-\mathbb{E}_I^\ast [D_th(X_t)])^{\rm T}{\rm d}I_t.
\end{equation}
Hence, alternatively, we consider the change of measure 
\begin{equation}
    \frac{{\rm d}\tilde{\mathbb{P}}}{{\rm d}\mathbb{P}}(X_t) = \tilde{D}_t^{-1},
\end{equation}
where 
\begin{equation}
{\rm d}\tilde{D}_t = \tilde{D}_t \,h(X_t)^{\rm T}
{\rm d}Z_t, \qquad \tilde{D}_0 = 1.
\end{equation}
We recall that this change of measure arises from an application of Girsanov's theorem \cite{Pavliotis2016} 
and that $Z_t$ behaves like Brownian motion independent of $B_t$ under the new path measure $\tilde{\mathbb{P}}$
\cite{bain2008fundamentals}.  We, hence, introduce a second pair of FSDEs
\begin{subequations} \label{eq:FSDEstilde}
    \begin{align} \label{eq:FSDE1tilde}
        {\rm d}X_t &=b(X_t){\rm d}t
        + \sigma {\rm d}B_t, \qquad  \qquad \qquad  X_0 \sim \mu,\\ \label{eq:FSDE2tilde}
        {\rm d} \tilde{D}_t &= \tilde{D}_t \,h(X_t)^{\rm T}{\rm d}Z_t, \qquad 
        \qquad  \qquad    \,\,\,\, \,
        \tilde{D}_0 = 1,
    \end{align}
\end{subequations}
which are  now driven by the independent Brownian processes $(B_t,Z_t)$. The induced path measure is denoted by $\tilde{\mathbb{P}}^\ast$ and its associated expectation operator by $\tilde{\mathbb{E}}^\ast$. 

Furthermore, one can introduce the non-normalized filtering density 
$\sigma_t$, which satisfies
\begin{equation}
\sigma_t[f] = \tilde{\mathbb{E}}^\ast [\tilde{D}_t f(X_t)|Z_{0:t}]
\end{equation}
for fixed observation process $Z_{0:t}$. Again we introduce the shorthand notation
\begin{equation}
\tilde{\mathbb{E}}^\ast_Z[g(X_t,\tilde{D}_t)]=\tilde{\mathbb{E}}^\ast[g(X_t,\tilde{D}_t)|Z_{0:t}].
\end{equation} 

Recall that $\sigma_t$ satisfies the Zakai equation of nonlinear filtering \cite{jazwinski2007stochastic,bain2008fundamentals} and that
\begin{equation}
\pi_t [f] = \frac{\sigma_t[f]}{\sigma_t[1]}.
\end{equation}

%
\section{FBSDE estimators for conditional expectation values}
\label{sec:estimation}
%

We wish to construct estimators for the conditional expectation values $\pi_T[f]$ and
$\sigma_T [f]$, respectively, of the following form:
\begin{subequations}
\begin{align} \label{eq:estimator I}
S_T^\sigma[f] &:= \tilde{Y}_0 - \int_0^T \tilde{U}_t^{\rm T} {\rm d}Z_t,\\ \label{eq:estimator II}
S_T^\pi[f] &:= \breve{Y}_0 - \int_0^T \breve{U}_t^{\rm T} {\rm d}I_t,\\ \label{eq:estimator III}
\hat{S}_T^\pi[f] &:= \hat{Y}_0 - \int_0^T \hat{U}_t^{\rm T} {\rm d}Z_t,\\
\label{eq:estimator IV}
\bar{S}_T^\sigma [f] &:= \bar{Y}_0 - \int_0^T \bar{U}_t^{\rm T}{\rm d}W_t.
\end{align}
\end{subequations}
The required random variables $\tilde{Y}_0$, $\breve{Y}_0$, $\hat{Y}_0$, $\bar{Y}_t$ and controls $\tilde{U}_{0:T}$, $\breve{U}_{0:T}$, $\hat{U}_{0:T}$, $\bar{U}_t$,
respectively, are chosen such that
\begin{subequations}
\begin{align} \label{eq:estimator I a}
\tilde{\mathbb{E}}_Z^\ast [S_T^\sigma[f]] &= \sigma_T[f],\\
\mathbb{E}_I^\ast [S_T^\pi[f]] &= \pi_T[f],\\
\mathbb{E}_I^\ast [\hat{S}_T^\pi[f] ]&= \pi_T[f], 
\label{eq:estimator III a}\\
\tilde{\mathbb{E}}_Z^\ast [\bar{S}_t^\sigma [f]] &= \sigma_T[f]
\end{align}
\end{subequations}
for any suitable observable $f(x)$. 

We derive appropriate FBSDE formulations for each of the four estimators in the following subsections. See \cite{Carmona,Bensoussan} for an introduction to BSDEs and FBSDEs.

%
\subsection{Observation-based estimator I} \label{sec:obs based I}
%

We start with the estimator (\ref{eq:estimator I}) and introduce the BSDE 
\begin{equation}
{\rm d}\tilde{Y}_t = \tilde{Q}_t^{\rm T} {\rm d}B_t +  
        \tilde{V}_t^{\rm T} {\rm d}Z_t, \qquad \quad \tilde{Y}_T = \tilde{D}_T f(X_T), \label{eq:BSDE1tilde}
\end{equation}
where  $(X_T,\tilde{D}_T)$ is defined by the FSDE (\ref{eq:FSDEstilde}).

We note that
\begin{subequations} 
\begin{align}
     \tilde{\mathbb{E}}^\ast_Z
\left[ \left( \tilde{D}_T f(X_T) - S_T^\sigma[f]\right)^2 \right] &=
 	    \tilde{\mathbb{E}}_Z^\ast \left[ \left( \tilde{D}_T f(X_T) - \tilde{Y}_0 +
    \int_0^T \tilde{U}_t^{\rm T}{\rm d}Z_t \right)^2 \right] \\
    &= \tilde{\mathbb{E}}_Z^\ast \left[ \left( \int_0^T \tilde{Q}_t^{\rm T} {\rm d}B_t +  
        \int_0^T (\tilde{U}_t+\tilde{V}_t)^{\rm T} {\rm d}Z_t  \right)^2 \right]\\
    &=  \tilde{\mathbb{E}}_Z^\ast \left[\int_0^T \left(\|\tilde{Q}_t\|^2 + \|\tilde{U}_t + \tilde{V}_t\|^2\right) {\rm d}t \right] + \tilde{M}_T,
\end{align}
\end{subequations}
where $\tilde{M}_T$ is a martingale term which vanishes under expectations with respect to observations $Z_{0:T}$. The desired control $\tilde{U}_{0:T}$ and its associated estimator (\ref{eq:estimator I}) are therefore chosen to minimize the cost functional
\begin{equation} \label{eq:cost_functional_sigma2}
    \tilde{J}_T(\tilde{U}_{0:T}) := 
     \tilde{\mathbb{E}}^\ast \left[\int_0^T \left(\|\tilde{Q}_t\|^2 + \|\tilde{U}_t + \tilde{V}_t\|^2\right) {\rm d}t \right] .
\end{equation}

The cost functional (\ref{eq:cost_functional_sigma2}) is minimized for
\begin{equation} \label{eq:optimal_control}
    \tilde{U}_t = -\tilde{V}_t.
\end{equation}
This control leads to an estimator (\ref{eq:estimator I}) of minimum variance as expressed by $\tilde{J}_T$. Furthermore, we immediately obtain from (\ref{eq:BSDE1tilde}) and (\ref{eq:estimator I}) that the mean estimation error satisfies
\begin{equation} \label{eq:estimator_error}
\sigma_T[f] - \tilde{\mathbb{E}}_Z^\ast [S_T^\sigma[f]]  = \tilde{\mathbb{E}}_Z^\ast \left[ \int_0^T (\tilde{U}_t+\tilde{V}_t)^{\rm T}{\rm d}Z_t \right] .
\end{equation}
Using (\ref{eq:optimal_control}), this implies (\ref{eq:estimator I a}) for the estimator (\ref{eq:estimator I}) as desired.


\noindent
In order to gain a better insight into the solution structure of the FBSDEs (\ref{eq:FSDEstilde}) and (\ref{eq:BSDE1tilde}) and to find an explicit explicit expression for the optimal control $\tilde{U}_{0:T}$, we introduce the 
notations $\tilde{Y}_s^{t,x,d}$, $\tilde{V}_s^{t,x,d}$, $\tilde{Q}_s^{t,x,d}$, and $\tilde{U}_s^{t,x,d}$ for solutions of the FBSDE considered over the restricted time interval $s\in [t,T]$ with initial condition 
$X_t = x$ and $\tilde{D}_t = d$. Equation (\ref{eq:BSDE1tilde}) now implies
\begin{equation} \label{eq:update1}
\tilde{Y}_t^{t,x,d} = \tilde{\mathbb{E}}^\ast \left[\tilde{Y}_{t+\tau}^{t,x,d}\right] 
\end{equation}
for $\tau\in (0,T-t]$, where $X_s^{t,x}$ denotes solutions of the forward SDE (\ref{eq:FSDE1tilde}) with $X_t = x$ for $s\ge t$ and similarly for $\tilde{D}_s^{t,x,d}$.  Here $\tilde{\mathbb{E}}^\ast$ denotes now expectation with respect to the Brownian noise $B_{t:t+\tau}$ in (\ref{eq:FSDE1tilde}) and the observations $Z_{t:t+\tau}$ in (\ref{eq:FSDE2tilde}) under the path measure $\tilde{\mathbb{P}}^\ast$ 
for fixed initial conditions $X_t = x$ and $\tilde{D}_t = d$. In particular, we can set $t+\tau = T$ and use 
$\tilde{Y}_T^{t,x,d} = \tilde{D}_T^{t,x,d} f(X_T^{t,x})$. 

Furthermore, let us denote the generator of the FSDEs (\ref{eq:FSDEstilde}) by $\tilde{\mathcal{L}}$, which acts on functions $g(x,d)$, that is,
\begin{equation}
\tilde{\mathcal{L}}\,g(x,d) = b(x)^{\rm T}\nabla_xg(x,d) + \frac{\sigma^2}{2}\nabla_x^2g(x,d)
+ \frac{d^2}{2} h(x)^{\rm T}h(x) \,\partial_d^2g(x,d).
\end{equation}
Then, under appropriate smoothness assumptions, the deterministic function 
\begin{equation}
\tilde{y}_t(x,d) := \tilde{Y}_t^{t,x,d}  
\end{equation}
 satisfies the backward Kolmogorov equation
\begin{equation}
-\partial_t \tilde{y}_t(x,d) = \tilde{\mathcal{L}}\,\tilde{y}_t(x,d), \qquad \tilde{y}_T(x,d) = d \,f(x),
\end{equation}
since
\begin{subequations}
\begin{align}
    {\rm d}\tilde{y}_t(X_t,\tilde{D}_t) &= (\partial_t \tilde{y}_t(X_t,\tilde{D}_t) +
    \tilde{\mathcal{L}} \tilde{y}_t(X_t,\tilde{D}_t)){\rm d}t \\
    &\qquad +\, \sigma \nabla_x \tilde{y}_t(X_t,\tilde{D}_t){\rm d}B_t +
    \partial_d \tilde{y}_t(X_t,\tilde{D}_t) \tilde{D}_t h(X_t)^{\rm T}{\rm d}Z_t \\&=  {\rm d}\tilde{Y}_t =
    \tilde{Q}_t^{\rm T} {\rm d}B_t + \tilde{V}_t^{\rm T}{\rm d} Z_t  .
\end{align}
\end{subequations}
Hence one also obtains
\begin{subequations}
    \begin{align}
        \tilde{Q}_t^{t,x,d} &= \sigma \nabla_x \tilde{y}_t(x,d),\\
        \tilde{V}_t^{t,x,d} &= d\,\partial_d \tilde{y}_t(x,d) \,h(x),
    \end{align}
\end{subequations}
and we can define the deterministic control function
\begin{equation}
\tilde{u}_t (x,d) := - d\,\partial_d \tilde{y}_t(x,d) \,h(x)
\end{equation}
and minimum variance control $\tilde{U}_t = \tilde{u}_t (X_t,\tilde{D}_t)$. 

To simplify further, we make the {\it ansatz} $\tilde{y}_t(x,d) = d \,y_t(x)$ and find that this assumption implies
\begin{equation}
    \tilde{\mathcal{L}} \,\tilde{y}_t(x,d) = d \,\mathcal{L}^x \,y_t(x),
\end{equation}
where $\mathcal{L}^x$ denotes the generator of (\ref{eq:FSDE1}).
Hence $y_t (x)$ has to satisfy the backward Kolmogorov equation
\begin{equation} \label{eq:BKE}
    -\partial_t y_t(x) = \mathcal{L}^x y_t(x), \qquad y_T(x) = f(x).
\end{equation}
Furthermore,
\begin{equation} \label{eq:simple1}
u_t (x,d) = -d \,y_t(x) \,h(x)
\end{equation}
and (\ref{eq:estimator I}) becomes
\begin{equation} 
    S_T^\sigma [f]  = y_0(X_0) + \int_0^T \tilde{D}_t \,y_t(X_t) \,h (X_t)^{\rm T}{\rm d}Z_t .
\end{equation}

A closely related estimator can be found in \cite{Kim22} (Section 2.3.3), where it has been used to derive Zakai's equation. The difference is that \cite{Kim22} considers directly the averaged estimator (\ref{eq:averaged estimator}) not relying on a FBSDE formulation. 

Furthermore, one can derive a BSDE representation for $y_t(x)$ directly. Upon introducing $Y_t$ such that $\tilde{Y}_t = \tilde{D}_t Y_t$ and utilizing (\ref{eq:FSDE2tilde}), (\ref{eq:output_n}), and (\ref{eq:BSDE1tilde}), one obtains the BSDE
\begin{equation} \label{eq:BSDE_W}
    {\rm d}Y_t =  Q_t^{\rm T} 
    {\rm d}B_t + V_t^{\rm T}{\rm d}W_t, \qquad Y_T = f(X_T),
\end{equation}
along the FSDEs (\ref{eq:FSDEs0}). 
Please be aware that $(Q_t,V_t)$ are not the same as in the BSDE (\ref{eq:BSDE1tilde}). In fact, since
\begin{subequations}
    \begin{align}
    {\rm d}\tilde{Y}_t &=\tilde{D}_t {\rm d}Y_t +
    Y_t \tilde{D}_t h(X_t)^{\rm T}{\rm d}Z_t +
    \tilde{D}_t h(X_t)^{\rm T}V_t {\rm d}t\\
    &= \tilde{D}_tQ_t^{\rm T}{\rm d}B_t + \tilde{D}_t \left( V_t + Y_t h(X_t)\right)^{\rm T}{\rm d}Z_t ,
    \end{align}
\end{subequations}
one finds that $\tilde{Q}_t = \tilde{D}_t Q_t$ and $\tilde{V}_t = \tilde{D}_t V_t + \tilde{D}_t Y_t h(X_t)$. Furthermore, since $Y_t$ does not depend on $Z_{t:T}$, $V_t \equiv 0$ and we obtain
\begin{equation} \label{eq:BSDE_B}
    {\rm d}Y_t = Q_t^{\rm T}{\rm d}B_t
\end{equation}
in line with the backward Kolmogorov equation (\ref{eq:BKE}). Let us summarize our findings in the following lemma.

\begin{lemma}
The estimator (\ref{eq:estimator I}) becomes unbiased for $\tilde{Y}_0 = Y_0$
with $Y_t$, $t\in [0,T]$, defined by the FBSDEs
\begin{subequations} \label{eq:FBSDE_BKE}
\begin{align}
    {\rm d}X_t &=b(X_t){\rm d}t + \sigma {\rm d}B_t, \quad    X_0 \sim \mu,\\
    {\rm d}Y_t &= Q_t^{\rm T}{\rm d}B_t, \qquad \qquad \,\,\, Y_T = f(X_T).
\end{align}
\end{subequations}
The control $\tilde{U}_t$, $t\in [0,T]$, is provided by
\begin{equation}
    \tilde{U}_t = -\tilde{D}_t Y_t h(X_t).
\end{equation}
\end{lemma}

\smallskip

\noindent
We end this section with a brief discussion related to the problem of filter stability. Let us introduce the abbreviation
\begin{equation}\label{eq:averaged estimator}
    \mathcal{S}_T^\sigma [f] := \tilde{\mathbb{E}}_Z^\ast [ S_T^\sigma[f]] = \mu[Y_0]-\int_0^T \tilde{\mathbb{E}}^\ast_Z[\tilde{U}_t]{\rm d}Z_t
\end{equation}
The optimal choice $\tilde U_t = - \tilde V_t$ leads to
\begin{equation}
    \tilde Y_T - \mathcal{S}_T^\sigma [f] = \tilde Y_0 - \mu[\tilde Y_0]
    + \int_0^T \tilde{Q}_t^{\rm T} {\rm d}B_t  +
    \int_0^T (\tilde V_t-\tilde{\mathbb{E}}_Z^\ast [\tilde V_t]) {\rm d}Z_t 
\end{equation}
from which we can recover
\begin{equation}
    \mathcal{S}_T^\sigma [f] = \tilde{\mathbb{E}}_Z^\ast [\tilde Y_T] = \sigma_T[f]
\end{equation}
(unbiasedness) as well as
\begin{subequations}
    \begin{align}
    \tilde{\mathbb{E}}^\ast \left[\left( \tilde Y_T - \sigma_T [f]\right)^2\right] &= 
    \mu \left[\left(\tilde Y_0 - \mu [\tilde Y_0] \right)^2 \right]
    \\ & \qquad \qquad +\,\tilde{\mathbb{E}}^\ast\left[\int_0^T \left( \|\tilde{Q}_t\|^2 + \|\tilde V_t-\tilde{\mathbb{E}}_Z^\ast[\tilde V_t]\|^2
    \right) {\rm d}t\right]
    \end{align}
\end{subequations}
(variance propagation). The last identity gives rise to the differential equation
\begin{equation} \label{eq:backward_variance_DE}
    \frac{\rm d}{{\rm d}t} \tilde{\mathbb{E}}^\ast \left[\left(\tilde Y_t -\sigma_t[f]\right)^2\right]=\tilde{\mathbb{E}}^\ast \left[   \|\tilde Q_t\|^2 + \| \tilde V_t - \tilde{\mathbb{E}}^\ast_Z[\tilde V_t]\|^2\right]
\end{equation}
and the variance of $\tilde Y_t$ is non-increasing as time goes backward from $t=T$ to $t=0$.
The expression  on the right hand side of (\ref{eq:backward_variance_DE}) can be viewed as a (non-stationary) Dirichlet form and {\it carr\'e du champ} operator \cite{book_BGL}, respectively, associated with the FSDE (\ref{eq:FSDEstilde}). Furthermore, under an assumed Poincar\'e-type inequality
\begin{equation}
   c \,\tilde{\mathbb{E}}^\ast \left[ \left( \tilde Y_t - \sigma_t[f]\right)^2\right] \le \tilde{\mathbb{E}}^\ast \left[\| \tilde Q_t\|^2 + \|\tilde V_t-\tilde{\mathbb{E}}^\ast_Z[\tilde V_t]\|^2\right] ,
\end{equation}
$c>0$, the decay becomes exponential. We note that 
\begin{equation}
    \|\tilde Q_t\|^2 + \| \tilde V_t - \tilde{\mathbb{E}}^\ast_Z[\tilde V_t]\|^2
    = \sigma^2 \|\tilde D_t \nabla_x y_t(X_t)\|^2 +\| \tilde D_t  y_t(X_t)h(X_t)-
    \sigma_t[y_t h] \|^2,
\end{equation}
which provides an explicit expression for the decay of variance via (\ref{eq:backward_variance_DE}). It would be of interest to relate these considerations to the problem of filter stability. See the closely related work \cite{KimMehta2023c}. An obvious observation is that an exponential decay in the variance of $Y_t$ implies an exponential decay of variance in $\tilde{Y}_t$.

%
\subsection{Innovation-based estimator} \label{sec:innovation}
%

We now turn to the estimator (\ref{eq:estimator II}). The required BSDE for $\breve{Y}_t$ is given by
\begin{equation} \label{eq:BSDE1}
    {\rm d}\breve{Y}_t = \breve{Q}_t^{\rm T}{\rm d}B_t + \breve{V}_t^{\rm T}{\rm d}I_t,  \quad\,\,\,  \,
    \breve{Y}_T = D_T f(X_T),
\end{equation}
along solutions of the FSDE (\ref{eq:FSDEs1}). We note again that
\begin{subequations} 
\begin{align}
    \mathbb{E}_I^\ast \left[ \left( D_T f(X_T) - S_T^\pi [f]\right)^2 \right]
    &=
    \mathbb{E}_I^\ast \left[ \left( D_T f(X_T) - \breve{Y}_0 +
    \int_0^T \breve{U}_t^{\rm T}{\rm d}I_t \right)^2 \right] \\
    &= \mathbb{E}_I^\ast \left[ \left( \int_0^T \breve{Q}_t^{\rm T} {\rm d}B_t +  
        \int_0^T (\breve{U}_t + \breve{V}_t)^{\rm T} {\rm d}I_t  \right)^2 \right]\\
    &= \mathbb{E}_I^\ast \left[ \int_0^T \left(\|\breve{Q}_t\|^2 + \| \breve{U}_t + \breve{V}_t\|^2\right) {\rm d}t \right] + \breve{M}_T,
\end{align}
\end{subequations}
where $\breve{M}_T$ is a martingale term which vanishes under expectation with respect to $I_{0:T}$. The control $\breve{U}_t$ in (\ref{eq:estimator II}) is chosen to minimize the cost functional
\begin{equation} \label{eq:cost_functional_pi2}
    \mathcal{J}_T(\breve{U}_{0:T}) :=  
     \mathbb{E}^\ast \left[ \int_0^T \left(\|\breve{Q}_t\|^2 + \| \breve{U}_t + \breve{V}_t\|^2\right) {\rm d}t \right],
\end{equation}
which is again minimized by $\breve{U}_t = - \breve{V}_t$ and the associated estimator (\ref{eq:estimator II}) is of minimum variance. 

The time-dependent generator $\mathcal{L}_t$ associated with the FSDE (\ref{eq:FSDEs1}) is provided by
\begin{equation}
\mathcal{L}_t \,g = b^{\rm T}\nabla_x g  
 + \frac{\sigma^2}{2}\nabla_x^2g
+ \frac{d^2}{2} \|h-\pi_t[h]\|^2 \,\partial_d^2g.
\end{equation}
Let the function $\breve{y}_t(x,d)$ satisfy the associated BSPDE
\begin{equation} \label{eq:BSPDE}
-{\rm d}_t  \breve{y}_t = \mathcal{L}_t \breve{y}_t \,{\rm d}t
- \breve{\mathcal{V}}_t^{\rm T} {\rm d}I_t, \qquad \breve{y}_T(x,d) = d \,f(x),
\end{equation}
where the $\breve{\mathcal{V}}_t$ term is needed to make $\breve{y}_t(x,d)$ adapted to the forward process. However, as shown below, $\partial_d^2 \breve{y}_t =0$ and we may conclude that $\breve{\mathcal{V}}_t \equiv 0$. Then It\^o's formula implies
\begin{subequations}
\begin{align}
    {\rm d} \breve{y}_t(X_t,D_t) &= ({\rm d}_t
    \breve{y}_t(X_t,D_t) +
    \mathcal{L}_t \breve{y}_t(X_t,D_t){\rm d}t) 
     \\
    &\qquad +\, \sigma \nabla_x \breve{y}_t(X_t,D_t){\rm d}B_t +
    D_t \partial_d \breve{y}_t(X_t,D_t)  (h(X_t)-\pi_t[h])^{\rm T}{\rm d}I_t \\&=
    \breve{Q}_t^{\rm T} {\rm d}B_t + \breve{V}_t^{\rm T}{\rm d} I_t = 
    {\rm d} \breve{Y}_t.
\end{align}
\end{subequations}
Hence, we find that $\breve{q}_t(x,d) = \sigma \nabla_x \breve{y}_t(x,d)$ as well as
\begin{equation}
    \breve{v}_t(x,d) =   d \,\partial_d \breve{y}_t(x,d)  \,(h(x)-\pi_t[h]).
\end{equation}
Again, we make the {\it ansatz} $\breve{y}_t(x,d) = d \,y_t(x)$ and find that $y_t(x)$ has to satisfy the previously stated backward Kolmogorov equation (\ref{eq:BKE}). Therefore, we also find that $\breve{Q}_t = \breve{q}_t(X_t,D_t)$ with
$\breve{q}_t(x,d) := d\,\sigma \,\nabla_x y_t(x)$ as well as $\breve{V}_t = \breve{v}_t(X_t,D_t)$ with $\breve{v}_t(x,d) = d\,v_t(x)$ where
\begin{equation}
v_t(x) = y_t(x) \, (h(x)-\pi_t[h]).
\end{equation}
The optimal control is given by
\begin{equation} \label{eq:oc1}
\breve{U}_t =  -D_t \,y_t(X_t)\,  (h(X_t)-\mathbb{E}_I^\ast [D_t h(X_t)])
\end{equation}
and (\ref{eq:estimator II}) becomes
\begin{equation} \label{eq:line444}
    S_T^\pi [f]  = y_0(X_0) + \int_0^T 
    D_t \,y_t(X_t)  \,(h(X_t)-\mathbb{E}_I^\ast [D_th(X_t)])^{\rm T}
    {\rm d}I_t .
\end{equation}

Alternatively, using the {\it ansatz} $\breve{Y}_t = D_t Y_t$, one can reformulate the BSDE (\ref{eq:BSDE1}) as a BSDE (\ref{eq:BSDE_W}) along the FSDE (\ref{eq:FSDEs0}) in the new variable $Y_t$. The two BSDEs are now related via $\breve{Q}_t= D_t Q_t$ as well as
\begin{equation}
    \breve{V}_t = D_t V_t + D_t Y_t \left(h(X_t)-\mathbb{E}_I^\ast [D_t h(X_t)]\right)
\end{equation}
and one finds again that $V_t \equiv 0$. Hence the BSDE (\ref{eq:BSDE_W}) reduces to (\ref{eq:BSDE_B}), which corresponds to the backward Kolmogorov equation (\ref{eq:BKE}). Let us summarize our findings in the following lemma.

\begin{lemma}
The estimator (\ref{eq:estimator II}) becomes unbiased for $\breve{Y}_0 = Y_0$
with $Y_t$, $t\in [0,T]$, defined by the FBSDEs (\ref{eq:FBSDE_BKE}) and
the control $\breve{U}_t$, $t\in [0,T]$, is provided by
\begin{equation}
    \breve{U}_t = -D_t Y_t (h(X_t)-\mathbb{E}_I^\ast[D_th(X_t)]).
\end{equation}
\end{lemma}

\smallskip

\noindent
The considerations on filter stability, as put forward in Section  \ref{sec:obs based I}, extend naturally to the innovation-based estimator (\ref{eq:estimator II}). In particular, (\ref{eq:backward_variance_DE}) becomes
\begin{subequations} \label{eq:backward_variance_DE2}
\begin{align}
    &\frac{\rm d}{{\rm d}t} \mathbb{E}^\ast \left[\left(\breve{Y}_t -\pi_t[f]\right)^2\right]\\
    &\qquad =\mathbb{E}^\ast \left[ 
    \sigma^2 \|D_t \nabla_x y_t(X_t)\|^2 +\| D_t y_t(X_t)(h(X_t)-\pi_t[h])-
    \pi_t[y_t (h-\pi_t[h])] \|^2
    \right]
\end{align}
\end{subequations}
with $\breve{Y}_t = D_t \,y_t(X_t)$ and an appropriate Poincar\'e inequality would again establish an exponential decay of variance.  

\begin{example} \upshape
The linear-Gaussian case is a special case of the model~(\ref{eq:FSDEs0}) where $b(x) = A^{\rm T}x$, $h(x) = H^{\rm T} x$ for some $A\in \Re^{n\times n}$ and $H\in \Re^{n\times m}$, and $\mu$ is Gaussian with mean $m_0$ and covariance matrix $\Sigma_0$. Fix $\bar{f} \in \Re^n$ and consider a linear function $f(x) = \bar{f}^{\rm T} x$. In this case, the solution to the backward Kolmogorov equation is also a linear function $y_t(x) = \bar{y}_t^{\rm T} x$ with
\begin{equation} \label{eq:LFPE}
    -\frac{{\rm d}\bar{y}_t}{{\rm d}t} = A\bar{y}_t,\quad \bar{y}_T = \bar{f}.
\end{equation}
Upon introducing the filtered state covariance matrix
 \begin{equation}
 \Sigma_t := \mathbb{E}_I^\ast[D_t X_tX_t^{\rm T}]-\mathbb{E}_I^\ast[D_t X_t]\,\mathbb{E}_I^\ast[D_t X_t]^{\rm T},
 \end{equation}
the optimal control (\ref{eq:oc1}) takes the form
\begin{equation}
    \mathbb{E}_I^*[U_t] = -H^{\rm T} \Sigma_t \bar{y}_t
\end{equation}
and the estimator~(\ref{eq:line444}) becomes
\begin{subequations}
\begin{align}
    \mathcal{S}_T^\pi[f] &:= \mathbb{E}_I^\ast [S_T^\pi[f]]\\
    &= \bar{y}_0^{\rm T}m_0 +\int_0^T \bar{y}_t^{\rm T}\Sigma_t H{\rm d}I_t\\
    &=\bar{f}^{\rm T}e^{T A^{\rm T}}m_0 +\int_0^T \bar{f}^{\rm T} e^{(T-t)A^{\rm T}}\Sigma_t H {\rm d}I_t
\end{align}
\end{subequations}
Define the filtered mean
 \begin{equation} \label{eq:LG-mean}
     m_t := e^{tA^{\rm T}}m_0 + \int_0^t e^{(t-s)A^{\rm T}}\Sigma_s H{\rm d}I_s,  
 \end{equation}
 and observe that $\mathcal{S}_T^\pi[f] = \bar{f}^{\rm T} m_T$. The Kalman--Bucy filter is obtained by differentiate~(\ref{eq:LG-mean}) with respect to $t$:
 \begin{equation}
     {\rm d} m_t = A^{\rm T} m_t {\rm d}t + \Sigma_t H{\rm d}I_t.
 \end{equation}
Please note that, contrary to (\ref{eq:LFPE}), the BSPDE formulations (\ref{eq:BSPDE_K}) leads to the controlled backward Kolmogorov equation
\begin{equation} \label{eq:LFPEU}
    -\frac{{\rm d}\bar{y}_t}{{\rm d}t} = A\bar{y}_t + H \bar{u}_t ,\quad \bar{y}_T = \bar{f}
\end{equation}
instead \cite{KimMehta2022b}. We explore this connection further in the following subsection. \hfill $\Box$
\end{example}
 
%
\subsection{Observation-based estimator II} \label{sec:obII}
%

We now state a BSDE required to define the estimator (\ref{eq:estimator III}). The key observation is again that the backward process $\hat{Y}_t$ is chosen such that 
\begin{equation}
D_T f(X_T) - \hat{S}^\pi_T[f] = \hat{Y}_T - \hat{Y}_0 + \int_0^T \hat{U}_t^{\rm T} {\rm d}Z_t 
\end{equation}
becomes independent of the innovation process $I_{0:T}$ under an optimal choice of the control $\hat{U}_{0:T}$. These considerations lead immediately to the BSDE
\begin{equation} \label{eq:BSDE3}
    {\rm d}\hat{Y}_t = \hat{Q}_t^{\rm T}{\rm d}B_t +
    \hat{V}_t^{\rm T}{\rm d}I_t - \hat{U}_t^{\rm T}\mathbb{E}_I^\ast [D_t h(X_t)]
    {\rm d}t, \qquad \hat{Y}_T = D_T f(X_T),
\end{equation}
along the FSDE (\ref{eq:FSDEs1}) since then
\begin{equation}
D_T f(X_T) - \hat{S}^\pi_T[f] = \int_0^T  \hat{Q}_t^{\rm T}{\rm d}B_t +
\int_0^T (\hat{U}_t + \hat{V}_t)^{\rm T}{\rm d}I_t 
\end{equation}
and the desired control is provided by $\hat{U}_t = -\hat{V}_t$. The same conclusion can be drawn from 
the associated cost function
\begin{equation}
    \hat{\mathcal{J}}_T(\hat{U}_{0:T}) := \mathbb{E}^\ast \left[ \int_0^T \left(\|\hat{Q}_t\|^2 + \| \hat{U}_t + \hat{V}_t\|^2\right) {\rm d}t \right] ,
\end{equation}
which arises from
\begin{subequations}
\begin{align}
      \mathbb{E}_I^\ast\left[ \left( D_T f(X_T) - \hat{S}_T^\pi[f]\right)^2 \right]
    &= \mathbb{E}_I^\ast \left[ \left( D_T f(X_T) - Y_0 +
    \int_0^T \hat{U}_t^{\rm T}{\rm d}I_t \right)^2 \right] \\
    &= \mathbb{E}_I^\ast \left[ \left( \int_0^T \hat{Q}_t^{\rm T} {\rm d}B_t +  
        \int_0^T (\hat{U}_t + \hat{V}_t)^{\rm T} {\rm d}I_t  \right)^2 \right]\\
    &= \mathbb{E}_I^\ast \left[ \int_0^T \left(\|\hat{Q}_t\|^2 + \| \hat{U}_t + \hat{V}_t\|^2\right) {\rm d}t \right] + \hat{M}_T
\end{align}
\end{subequations}
and we have again demonstrated the minimum variance property of the estimator (\ref{eq:estimator III}) up to the martingale term $\hat{M}_T$.

However, contrary to the previous two estimators, the BSDE (\ref{eq:BSDE3}) does no longer lead to a standard backward Kolmogorov-type PDE. This is due to the appearance of the 
\begin{equation}
\hat{U}_t^{\rm T}\mathbb{E}_I^\ast [D_th(X_t)] = -\hat{V}_t^{\rm T}\mathbb{E}_I^\ast [D_th(X_t)]
\end{equation}
drift term in (\ref{eq:BSDE3}). Let us summarize our findings in the following lemma.

\begin{lemma}
The estimator (\ref{eq:estimator III}) becomes unbiased for $\hat{Y}_0$
defined by the FBSDEs
\begin{subequations} 
\begin{align}
        {\rm d}X_t &=b(X_t){\rm d}t + \sigma {\rm d}B_t, \qquad  \qquad \qquad  \qquad \qquad\, \,X_0 \sim \mu,\\ 
        {\rm d} D_t &= D_t \,(h(X_t)-\mathbb{E}_I^\ast[D_t h(X_t)])^{\rm T}{\rm d}I_t, 
        \qquad  \qquad   \, D_0 = 1,\\
    {\rm d}\hat{Y}_t &= \hat{Q}_t^{\rm T}{\rm d}B_t + \hat{V}_t^{\rm T} ({\rm d}I_t + \mathbb{E}_I^\ast[D_th(X_t)]{\rm d}t), \qquad  \hat{Y}_T = D_T f(X_T).
    \end{align}
\end{subequations}
The control $\hat{U}_t$, $t\in [0,T]$, is provided by $\hat{U}_t = -\hat{V}_t$. 
\end{lemma}

\smallskip

\noindent
Let us again discuss some implications of the proposed estimator and
introduce the abbreviation
\begin{equation}
    \hat{\mathcal{S}}_T^\pi[f] := \mathbb{E}_I^\ast [ \hat{S}_T^\pi[f]] =\mu[\hat{Y_0}] -\int_0^T \mathbb{E}_I^\ast [\hat{U}_t]^{\rm T} {\rm d}Z_t,
\end{equation}
which implies
\begin{subequations} \label{eq:ttt}
    \begin{align}
    \hat Y_T - \hat{\mathcal{S}}_T^\pi[f] &= \hat Y_0 - \mu[\hat Y_0]
    + \int_0^T \hat{Q}_t^{\rm T} {\rm d}B_t \\
    & \qquad \qquad +\,
    \int_0^T (\hat V_t+\mathbb{E}_I^\ast [\hat U_t])^{\rm T} {\rm d}I_t + \int_0^T
    (\hat U_t -\mathbb{E}_I^\ast [ \hat U_t])^{\rm T} \pi_t[h]{\rm d}t
    \end{align}
\end{subequations}
The choice $\hat U_t = - \hat V_t$ leads to
\begin{equation}
    \hat{\mathcal{S}}_T^\pi[f] = \mathbb{E}_I^\ast [\hat Y_T] = \pi_T[f]
\end{equation}
as well as
\begin{subequations} \label{eq:mv5}
    \begin{align}
    \mathbb{E}^\ast \left[\left( \hat Y_T - \hat{\mathcal{S}}_T^\pi[f] \right)^2\right] &= 
    \mu \left[\left(\hat Y_0 - \mu [\hat Y_0] \right)^2 \right]\\ &\qquad \qquad
    + \, \mathbb{E}^\ast \left[ \int_0^T \left( \|\hat{Q}_t\|^2 + \|\hat V_t-\mathbb{E}_I^\ast [\hat V_t]\|^2
    \right)
    {\rm d}t \right] \\ \label{eq:variance reduction 3}
    &\qquad \qquad +\, \mathbb{E}^\ast \left[ \left(\int_0^T
    (\hat V_t -\mathbb{E}_I^\ast [\hat V_t])^{\rm T} \pi_t[h]{\rm d}t \right)^2 \right] .
    \end{align}
\end{subequations}
Furthermore, if one applies the averaged control
\begin{equation} \label{eq:oc3}
    \hat{U}_t = -\mathbb{E}_I^\ast [\hat{V}_t]
\end{equation}
in the BSDE (\ref{eq:BSDE3}) directly, then (\ref{eq:ttt}) turns into
\begin{equation}
    \hat Y_T - \hat{\mathcal{S}}_T^\pi[f] = \hat Y_0 - \mu[\hat Y_0]
    + \int_0^T \hat{Q}_t^{\rm T} {\rm d}B_t +
    \int_0^T (\hat V_t-\mathbb{E}_I^\ast [\hat V_t])^{\rm T} {\rm d}I_t .
\end{equation}
Hence the integral term (\ref{eq:variance reduction 3}) vanishes and the resulting (\ref{eq:mv5}) reveals the more common minimum variance/optimal control aspect of the averaged control (\ref{eq:oc3}). We also note that the resulting estimator (\ref{eq:estimator III}) remains unbiased under the averaged control (\ref{eq:oc3}). 

Let us finally connect the discussion so far to the work 
of \cite{Kim22,KimMehta2022b}, which is based on the estimator (\ref{eq:est})
and the BSPDE (\ref{eq:BSPDE_K}). We note that $Z_t$ is not Brownian motion with respect to the path measure $\mathbb{P}$ generated by the FSDEs (\ref{eq:state_n})-(\ref{eq:output_n}). Hence we first reformulate (\ref{eq:BSPDE_K}) into
the BSDE 
\begin{equation} \label{eq:BSDE3c}
    {\rm d}Y_t = Q_t^{\rm T}{\rm d}B_t +
    V_t^{\rm T}{\rm d}W_t - \mathcal{U}_t^{\rm T}h(X_t)\,
    {\rm d}t, \qquad Y_T = f(X_T).
\end{equation}
along the FSDE (\ref{eq:FSDEs0}). In a next step, we obtain
\begin{equation}
    Y_T - \mathcal{S}_T 
    =Y_0 - \mu[Y_0] + \int_0^T
    Q_t^{\rm T}{\rm d}B_t + \int_0^T (V_t + \mathcal{U}_t)^{\rm T} {\rm d}W_t
\end{equation}
and
\begin{equation} \label{eq:cost_functiona_optimal_control}
    \mathbb{E} \left[ (Y_T-\mathcal{S}_T)^2 \right] = 
    \mathbb{E}\left[(Y_0-\mu[Y_0])^2 \right] + \mathbb{E} \left[
    \int_0^T \left( \|Q_t\|^2 + \|V_t + \mathcal{U}_t\|^2 \right) {\rm d}t
    \right] .
\end{equation}
Please note that the optimal control is {\it not} provided by
$\mathcal{U}_t = -\mathbb{E}[V_t|Z_{0:t}]$. In fact, we introduce
\begin{equation}
\check{Y}_t := D_t Y_t
\end{equation}
and, using (\ref{eq:FSDEs1}) and (\ref{eq:BSDE3c}), obtain
the BSDE
\begin{equation}\label{eq:1b}
{\rm d} \check{Y}_t  = \check{Q}_t^{\rm T}{\rm d}B_t + \check{V}_t^{\rm T}{\rm d}I_t
- D_t \,\mathcal{U}_t^{\rm T}h(X_t){\rm d}t, \qquad \check{Y}_T = D_T f(X_T),
\end{equation}
with $\check{V}_t$ and $\check{Q}_t$ satisfying
\begin{equation}
\check{V}_t =  D_t \left\{V_t + Y_t (h(X_t) - \pi_t[h]) \right\}
\end{equation}
and $\check{Q}_t = D_t Q_t$, respectively.
Note that the BSDE (\ref{eq:1b}) implies
\begin{equation} \label{eq:backward_estimator}
    \pi_T[f] = \mu[\check{Y}_0] +
    \int_0^T \mathbb{E}^\ast_I [\check{V}_t]^{\rm T}
    {\rm d}I_t - \int_0^T \mathcal{U}_t^{\rm T} \pi_t[h] {\rm d}t
\end{equation} 
and we can characterize the dependence of the estimation error on $\mathcal{U}_t$:
\begin{equation} 
\mathcal{S}_T - \pi_T[f]  = -\int_0^T \left( \mathcal{U}_t + \mathbb{E}^\ast_I 
\left[\check{V}_t\right]\right)^{\rm T} {\rm d}I_t.
\end{equation}
The error becomes zero (unbiased estimator) for
\begin{equation} \label{eq:oc2}
    \mathcal{U}_t =  -
    \mathbb{E}_I^\ast [\check{V}_t] =
    -\mathbb{E}_I^\ast [D_t V_t + D_t Y_t (h(X_t) - \pi_t[h])],
\end{equation}
which also provides the minimizer of (\ref{eq:cost_functiona_optimal_control}) (minimum variance estimator) in agreement with the results from \cite{KimMehta2022b}. 

Please also compare the BSDE formulations (\ref{eq:BSDE3}) to the BSDE formulation (\ref{eq:1b}), which are both along the FSDE (\ref{eq:FSDEs1}). In particular, $\hat{Q}_t$ corresponds formally to $\check{Q}_t$ and $\hat{V}_t$ to $\check{V}_t$, respectively. There remains a difference in the use of either $\hat{U}_t^{\rm T}\mathbb{E}_I^\ast[D_t h(X_t)]$ or $\mathcal{U}_t^{\rm T}(D_t h(X_t))$, respectively, as the additional drift term. However, since the representation of an unbiased estimator is unique (Proposition 2.31 in \cite{bain2008fundamentals}), the resulting estimators are equivalent, that is, $\hat{\mathcal{S}}_T^\pi[f] = \mathcal{S}_T[f]$.

\begin{example} \upshape
In the linear Gaussian case, $V_t \equiv 0$ in (\ref{eq:BSDE3c}) and the BSDE gives rise to the controlled backward Kolmogorov equation (\ref{eq:LFPEU}). 
Furthermore, the control (\ref{eq:oc2}) reduces to
\begin{equation}
    \bar{u}_t = -H^{\rm T} \Sigma_t \bar{y}_t,
\end{equation}
where $\bar{y}_t$ now satisfies the closed loop backward equation 
\begin{equation}
    -\frac{{\rm d}\bar{y}_t}{{\rm d}t} = A\bar{y}_t - H H^{\rm T}
    \Sigma_t \bar{y}_t ,\quad \bar{y}_T = \bar{f}
\end{equation}
in line with the dual optimal control perspective of \cite{kalman1960new,kalman1961new,Kim22}. \hfill $\Box$
\end{example}

%
\subsection{Observation error based estimator}
%

We finally discuss the rather unconventional estimator (\ref{eq:estimator IV}) with the observation error defined by (\ref{eq:obs_error}).
We again employ the FSDE (\ref{eq:FSDEstilde}) and introduce the associated BSDE
\begin{equation}
    {\rm d}\bar Y_t = \bar{Q}_t^{\rm T} {\rm d}B_t + \bar{V}_t^{\rm T}{\rm d}Z_t
    + \bar{U}_t^{\rm T}h(X_t){\rm d}t, \qquad \bar{Y}_T = \tilde D_T f(X_T).
\end{equation}
Hence
\begin{equation}
    \bar{Y}_T - \bar{S}_T^\sigma[f] = 
    \int_0^T \bar{Q}_t^{\rm T}{\rm d}B +
    \int_0^T (\bar{V}_t+ \bar{U}_t)^{\rm T}{\rm d}Z_t
\end{equation}
from which we conclude that
\begin{equation}
     \tilde{\mathbb{E}}^\ast_Z[\bar{S}_T^\sigma[f]]= \tilde{\mathbb{E}}^\ast_Z[\bar Y_T] = \sigma_T[f]
\end{equation}
provided that $\bar{U}_t = - \bar{V}_t$. 

Let us introduce a second BSDE for $Y_t$ defined by $\bar Y_t = \tilde D_t Y_t$. The BSDE is given by
\begin{equation}
    {\rm d}Y_t = Q_t^{\rm T} {\rm d}B_t + V_t^{\rm T}{\rm d}W_t
    + U_t^{\rm T}h(X_t){\rm d}t, \qquad Y_T = f(X_T),
\end{equation}
and the following identities hold:
\begin{equation}
    \bar{Q}_t = \tilde D_t Q_t, \quad
    \bar{V}_t = \tilde{D}_t (V_t + Y_t h(X_t)), \quad
    \bar{U}_t = \tilde{D}_t U_t.
\end{equation}
Using the optimal control $\bar{U}_t = - \bar{V}_t$, we obtain the closed
FBSDE system consisting of the FSDE (\ref{eq:FSDEs0}) and the BSDE
\begin{equation}
    {\rm d}Y_t = Q_t^{\rm T} {\rm d}B_t + V_t^{\rm T}{\rm d}W_t
    - (V_t + Y_t h(X_t))^{\rm T}h(X_t){\rm d}t, \qquad Y_T = f(X_T).
\end{equation}
It holds that $V_t \equiv 0$ and one obtains the Feynman-Kac type BSDE \cite{Pavliotis2016,Carmona}
\begin{equation}
    {\rm d}Y_t = Q_t^{\rm T} {\rm d}B_t - Y_t h(X_t)^{\rm T}h(X_t){\rm d}t, \qquad Y_T = f(X_T)
\end{equation}
along the FSDE (\ref{eq:state_n}) with associated backward PDE
\begin{equation} \label{eq:Feynman-Kac}
    -\partial_t y_t = \mathcal{L}^x y_t - \|h\|^2 y_t, \qquad y_T = f.
\end{equation}
Let us summarize our findings in the following lemma.

\begin{lemma}
The estimator (\ref{eq:estimator IV}) becomes unbiased for $\bar{Y}_0$
defined by $\bar{Y}_0 = y_0(X_0)$, where $y_t$ satisfies the backward PDE (\ref{eq:Feynman-Kac}) and the control $\bar{U}_t$, $t\in [0,T]$, is provided by 
\begin{equation}
\bar{U}_t = -\tilde{D}_t y_t(X_t) h(X_t).
\end{equation}
\end{lemma}

%
\section{Application: Optimal control of partially observed diffusion} \label{sec:SOC}
%

In this section, we discuss an application of the estimator (\ref{eq:estimator II})
to partially observed stochastic optimal control problems \cite{Bensoussan92,Borkar}. More specifically, consider the controlled diffusion process
\begin{equation} \label{eq:controlled_SDE}
{\rm d}X_t = b(X_t,\alpha_t){\rm d}t + \sigma {\rm d}B_t
\end{equation}
subject to minimising the cost function
\begin{equation} \label{eq:OC_cost_full}
\mathcal{V}_T(\alpha_{0:T}) 
= \mathbb{E}\left[\int_0^T c_t(X_t,\alpha_t) {\rm d}t
+ f(X_T) \right]
\end{equation}
for given $c_t(x,\alpha)$ over all admissible controls $\alpha_{0:T}$.
Since we wish to condition the control on the available data, we introduce the separated cost function
\begin{equation} \label{eq:OC_cost}
\mathcal{V}_T(\alpha_{0:T}|I_{0:T}) 
= \mathbb{E}_I^\ast \left[ \int_0^T D_t c_t(X_t,\alpha_t){\rm d}t +
D_T f(X_T)\right] .
\end{equation}
Note that (\ref{eq:OC_cost}) corresponds to the conditional expectation value (\ref{eq:general_expectation}) for fixed control.

\begin{example} \upshape \label{ex:control}
A standard example is provided by adding an additive control, that is,
\begin{equation}
    b(x,\alpha) = \tilde{b}(x) + G \alpha,
\end{equation}
where $G$ denotes an appropriate matrix and $\tilde{b}$ a control independent drift, and applying a quadratic cost function $c_t(x,\alpha) = \frac{1}{2} \|\alpha\|^2$. \hfill $\Box$
\end{example}

\noindent
Upon extending the techniques developed in Section \ref{sec:innovation}, we construct an estimator of the form
\begin{equation}  \label{eq:cost_estimagor}
    \mathcal{V}_T(\alpha_{0:T}|I_{0:T}) = \mu[\breve{Y}_0] + \int_0^T \mathbb{E}_I^\ast [ \breve{V}_t]^{\rm T}{\rm d}I_t,
\end{equation}
where $(\breve{Y}_t,\breve{V}_t)$ satisfy an appropriate generalization of the BSDE (\ref{eq:BSDE1}), that is,
\begin{equation} \label{eq:OC_estimator}
{\rm d}\breve{Y}_t = \breve{Q}_t^{\rm T}{\rm d}B_t + \breve{V}_t^{\rm T}{\rm d}I_t - D_t c_t(X_t,\alpha_t){\rm d}t, \qquad \breve{Y}_T = D_T f(X_T),
\end{equation}
along the FSDE (\ref{eq:FSDEs1}) with the drift term in (\ref{eq:FSDE1}) now being given by $b(X_t,\alpha_t)$.
Note that $\mu[\breve{Y}_0] = \mathcal{V}_T(\alpha_{0:T})$. 

Following the weak formulation of stochastic optimal control \cite{Bensoussan,Carmona}, we introduce 
\begin{equation}
\bar{B}_t = B_t + \sigma^{-1} \int_0^t b(X_s,\alpha_s){\rm d}s
\end{equation}
and recall that $\bar{B}_t$ is Brownian motion under a modified probability measure $\bar{\mathbb{P}}^\ast$ according to Girsanov's theorem \cite{Pavliotis2016}. Hence the FSDEs
(\ref{eq:FSDE1})-(\ref{eq:FSDE2}) become
\begin{subequations} \label{eq:FSDE_alpha}
    \begin{align} \label{eq:FSDE1_alpha}
        {\rm d}\bar{X}_t &= \sigma {\rm d}\bar{B}_t, \qquad  \qquad \qquad  \qquad \qquad \qquad \qquad \qquad 
        \qquad \quad  \bar{X}_0 \sim \mu,\\ \label{eq:FSDE2_alpha}
        {\rm d} \bar{D}_t &= \bar{D}_t \,(h(\bar{X}_t)-\pi_t[h])^{\rm T}{\rm d}I_t + \sigma^{-1} \bar{D}_t \, b(\bar{X}_t,\alpha_t)^{\rm T} {\rm d}\bar{B}_t, \quad \,\,
             \bar{D}_0 = 1.
     \end{align}
\end{subequations}
The associated BSDE is provide by
\begin{equation}\label{eq:OC_estimator_alpha}
{\rm d}\bar{Y}_t = \bar{Q}_t^{\rm T}{\rm d}\bar{B}_t + \bar{V}_t^{\rm T}{\rm d}I_t - 
\bar{D}_t c_t(\bar{X}_t,\alpha_t) {\rm d}t, \quad \bar{Y}_T = \bar{D}_T f(\bar{X}_T).
\end{equation}
In order to find the desired optimal control, we make the {\it ansatz} $\bar{Y}_t = \bar{D}_t Y_t$, which turns the BSDE (\ref{eq:OC_estimator_alpha}) into the BSDE
\begin{equation} \label{eq:BSDE_OC}
{\rm d}Y_t = Q_t^{\rm T} {\rm d}\bar{B}_t + V_t^{\rm T}{\rm d}W_t - \left\{ c_t(\bar{X}_t,\alpha_t)+ \sigma^{-1} Q_t^{\rm T} b(\bar{X}_t,\alpha_t) \right\}{\rm d}t, \quad Y_T = f(\bar{X}_T),
\end{equation}
along the FSDEs (\ref{eq:FSDE1_alpha}) and 
\begin{align} \label{eq:FSDE3_alpha}
    {\rm d} \bar{D}_t &= \bar{D}_t \left\{
    (h(\bar{X}_t)-\pi_t[h])^{\rm T}{\rm d}W_t 
    + \|h(\bar{X}_t)-\pi_t[h]||^2{\rm d}t
    + \sigma^{-1} b(\bar{X}_t,\alpha_t)^{\rm T} {\rm d}\bar{B}_t
    \right\}.
\end{align}
Compare the previously derived BSDE (\ref{eq:BSDE_W}) and the discussion from Section \ref{sec:innovation}. Recall that $\mu[\breve{Y}_0]=\mu[\bar{Y}_0] = \mu[Y_0]$, $\bar{Q}_t = \bar{D}_t Q_t$, and $\bar{V}_t = \bar{D}_t (V_t + Y_t (h(X_t)-\pi_t[h])$.

Let us now consider the case of optimal Markov policies/state controls of the form $\alpha_t = a_t(\bar{X}_t)$ for suitable functions $a_t(x)$. Upon applying the comparison theorem \cite{Carmona} to (\ref{eq:BSDE_OC}) and using the fact that $\bar{X}_t$ is independent of the chosen control, we obtain
\begin{equation} \label{eq:optimlity}
    \alpha_t = \arg \min_{\alpha}
    \left\{ c_t(\bar{X}_t,\alpha) + \sigma^{-1} 
    Q_t^{\rm T} b(\bar{X}_t,\alpha)
    \right\} .
\end{equation}

\begin{example} \upshape 
Following the setting from Example \ref{ex:control}, we
obtain the optimal control
\begin{equation}
\alpha_t = - \sigma^{-1} G^{\rm T}Q_t.
\end{equation}
The final step is to determine $Q_t$ as a function of $\bar{X}_t$. \hfill $\Box$
\end{example}

\noindent
We note that the optimal $\alpha_t$ does not depend on $\bar D_t$
and, hence, $V_t \equiv 0$ in (\ref{eq:BSDE_OC}) implying that the FSDE (\ref{eq:FSDE3_alpha}) can be dropped when considering 
the BSDE (\ref{eq:BSDE_OC}). Hence, the desired state control is of the form $\alpha_t = a_t (\bar{X}_t)$ with $a_t(x)$ an appropriate deterministic function defined in terms of $Q_t = \sigma \nabla_x y_t(\bar{X}_t)$, where $y_t(x)$ satisfies the well-known Hamilton--Jacobi--Bellman (HJB) equation \cite{Borkar,Bensoussan92,Carmona}
\begin{equation} \label{eq:bPDF}
-\partial_t y_t (x) = \mathcal{L}^{x,\alpha_t} y_t(x) + c_t(x,a_t(x)), \qquad y_T(x) = f(x).
\end{equation}
Here $\mathcal{L}^{x,\alpha_t}$ denotes now the generator of the SDE (\ref{eq:controlled_SDE}) with feedback control $\alpha_t = a_t(x)$. It holds that  $\mu[Y_0] = \mu[y_0]$.

\begin{lemma} The optimal state control law, which yields the optimal 
(\ref{eq:OC_cost}) along given data $Z_{0:T}$, is provided by
\begin{equation} \label{eq:AT}
    a_t(x) = \arg \min_{\alpha}
    \left\{ c_t(x,\alpha) + (\nabla_x y_t(x))^{\rm T} 
    b(x,\alpha)    \right\},
\end{equation}
where $y_t(x)$ satisfies the HJB equation (\ref{eq:bPDF}).
\end{lemma}

\noindent
In other words, the stochastic optimal control problem with cost (\ref{eq:OC_cost_full}) has been reduced to a data-independent PDE/FBSDE formulation discussed in detail in \cite{Carmona}. Once the optimal state control $\alpha_t = a_t(X_t)$ has been found, it can be used in (\ref{eq:controlled_SDE}) and the associated filtering problem can be solved using, for example, the feedback particle filter \cite{meyn13,CRS22,taghvaei2023survey}. In summary, it has been demonstrated that the problem of finding the optimal data-dependent state control is separable \cite{Wittenmark75}. 

We emphasize that the optimal state control problem discussed  so far is different from the one usually considered for partially observed Markov processes where the control $\alpha_t$ should depend on the previously observed data $Z_{0:t}$ only \cite{Borkar,Bensoussan92}. Following the certainty equivalence principle \cite{WW81,Wittenmark75}, a data-dependent control is provided by
\begin{equation}
\alpha_t = \pi_t [a_t],
\end{equation}
where $a_t(x)$ is given by the optimal state-dependent control (\ref{eq:AT}). The thus defined control is optimal for linear Gaussian processes \cite{WW81,Bensoussan92}. 

More generally, optimality requires to replace (\ref{eq:optimlity}) by a direct minimization of (\ref{eq:cost_estimagor}) over all admissible controls 
$\alpha_{0:T}$. It is then to be expected that $V_t\not=0$ in the BSDE
(\ref{eq:BSDE_OC}), in general, and the backward PDE (\ref{eq:bPDF}) gets replaced by the BSPDE 
\begin{equation} \label{eq:BSPDE_OC}
-{\rm d} {\rm Y}_t  = \mathcal{L}^{x,\alpha_t} {\rm Y}_t {\rm d}t   + c_t(\cdot ,\alpha_t){\rm d}t -
{\rm V}_t^{\rm T} {\rm d}W_t, \qquad {\rm Y}_T = f.
\end{equation}
This SPDE formulation leads to a dual optimal control problem \cite{Wittenmark75} with the forward SPDE provided by the KS equation 
\begin{equation} \label{eq:KSE}
{\rm d}\pi_t = (\mathcal{L}^{x,\alpha_t})^\dagger \pi_t {\rm d}t +
\pi_t (h-\pi_t[h])^{\rm T}{\rm d}I_t
\end{equation}
\cite{jazwinski2007stochastic,bain2008fundamentals}. Here $(\mathcal{L}^{x,\alpha_t})^\dagger$ denotes the adjoint of $\mathcal{L}^{x,\alpha_t}$. Upon calculating $
{\rm d}\pi_t[{\rm Y}_t]$, it is easy to verify that
\begin{equation} \label{eq:cost_estimator_PDE}
\mathcal{V}_T(\alpha_{0:T}|I_{0:T}) = \mu[{\rm Y}_0]
+\int_0^T \pi_t\left[{\rm V}_t + {\rm Y}_t (h - \pi_t[h]) \right]^{\rm T} {\rm d}I_t,
\end{equation}
which corresponds to (\ref{eq:cost_estimagor}). See Eq.~8.2.40 in \cite{Bensoussan92} for a related BSPDE formulation associated with Zakai's evolution equation for the unnormalized filtering distribution $\sigma_t$. 

\begin{lemma} The control in the KS equation (\ref{eq:KSE}) and the BSPDE (\ref{eq:BSPDE_OC}) has to satisfy
\begin{equation} \label{eq:optimal_alpha}
\alpha_t = \arg \min_\alpha \pi_t\left[
\mathcal{L}^{x,\alpha} {\rm Y}_t    + c_t(\cdot ,\alpha) \right]
\end{equation}
in order to minimize $\mathcal{V}_T(\alpha_{0:T}) = \mu[{\rm Y}_0]$.
\end{lemma}

\begin{proof} See Theorem 8.2.1.~in \cite{Bensoussan92} for the related optimality criterion in terms of Zakai's equation and its adjoint BSPDE formulation. A formal argument can be stated as follows. Let $\alpha_{0:T}^{(0)}$ denote some control and $\pi_{0:T}^{(0)}$, ${\rm Y}_{0:T}^{(0)}$ the associated solutions of
the KS equation (\ref{eq:KSE}) and the BSPDE (\ref{eq:BSPDE_OC}), respectively. Assume that the control does not satisfy (\ref{eq:optimal_alpha}). Then determine a new ${\rm Y}_{0:T}^{(1)}$ by solving (\ref{eq:BSPDE_OC}) subject to a new control law $\alpha_{0:T}^{(1)}$ determined by (\ref{eq:optimal_alpha}) with $\pi_{0:T} = \pi_{0:T}^{(0)}$. It holds by the comparison theorem that $\mu[{\rm Y}_0^{(1)}] < \mu[{\rm Y}_0^{(0)}]$. Next determine $\pi_{0:T}^{(1)}$ by solving (\ref{eq:KSE}) with the new control law $\alpha_{0:T}^{(1)}$. This procedure is to be repeated until a control $\alpha_{0:T}^\ast$ has been found which cannot be improved upon further.
\end{proof}

\begin{example}
Let us return to Example \ref{ex:control}. The optimal control satisfies
\begin{equation} \label{eq:exOC}
\alpha_t = -G^{\rm T} \pi_t[\nabla_x {\rm Y}_t].
\end{equation}
The function ${\rm Y}_{0:T}$ and the filtering distribution 
$\pi_{0:T}$ have to be determined self consistently using the BSPDE
(\ref{eq:cost_estimator_PDE}) and the KS equation
(\ref{eq:KSE}), respectively.
\end{example}

\noindent
The SPDE pair (\ref{eq:KSE}) and (\ref{eq:BSPDE_OC}) can be replaced by the FSDEs (\ref{eq:FSDE1_alpha}) and (\ref{eq:FSDE3_alpha}) together with the BSDE (\ref{eq:BSDE_OC}). The optimility condition (\ref{eq:optimal_alpha}) turns then into
\begin{equation} \label{eq:optimality_data}
    \alpha_t = \arg \min_{\alpha} \bar{\mathbb{E}}_I^\ast
    \left\{ \bar{D}_t c_t(\bar{X}_t,\alpha) + \sigma^{-1} 
    \bar{D}_t Q_t^{\rm T} b(\bar{X}_t,\alpha) \right] .
\end{equation}
In other words, the derivations of state- and data-dependent optimal controls differ only in the choice of the minimizing control, that is, by obeying either (\ref{eq:optimlity}) or (\ref{eq:optimality_data}) along the same set of FBSDEs. It will be of interest to  explore numerical approximations along the lines of \cite{CKSY21,EHJ17,EHJ21}.

\begin{remark}
We close this section by highlighting a link to the observation-based estimator (\ref{eq:est}) and its derivation in Section \ref{sec:obII}. Let us formally introduce the following partially observed stochastic control problem. We set $b(x,\alpha) = b(x)$ in (\ref{eq:controlled_SDE}) and introduce the running cost 
\begin{equation}\label{eq:cost_optimal_estimator}
    c_t(x,\alpha) = h(x)^{\rm T}\alpha.
\end{equation}
Hence the cost functional (\ref{eq:OC_cost}) becomes
\begin{equation} \label{eq:dddd}
\mathcal{V}_T(\alpha_{0:T}|I_{0:T}) 
= \mathbb{E}_I^\ast \left[ \int_0^T D_t h(X_t)^{\rm} \alpha_t {\rm d}t +
D_T f(X_T)\right] 
\end{equation}
and its estimator is given by (\ref{eq:cost_estimagor}) and (\ref{eq:cost_estimator_PDE}), respectively. Then (\ref{eq:cost_estimator_PDE}) and (\ref{eq:dddd}) together imply
\begin{equation}
\pi_T[f] = \mu[{\rm Y}_0] - \int_0^T \pi_t[h]^{\rm T}\alpha_t {\rm d}t
+ \int_0^T \pi_t\left[{\rm V}_t + {\rm Y}_t (h - \pi_t[h]) \right]^{\rm T} {\rm d}I_t.
\end{equation}
The definition (\ref{eq:innovation}) of the innovation process dictates
that the estimator (\ref{eq:est}) corresponds to the choice
\begin{equation}
    \mathcal{U}_t = \alpha_t = -\pi_t\left[{\rm V}_t + {\rm Y}_t (h - \pi_t[h]) \right].
\end{equation}
However, please be aware that $\alpha_{0:T} = \mathcal{U}_{0:T}$ is not a minimizer of the cost functional (\ref{eq:dddd}) \cite{KimMehta2022a}. 
\end{remark}

%
\section{Conclusions}
%

Building on the previous work \cite{Kim22,KimMehta2022a,KimMehta2022b} on optimal estimation for nonlinear filtering problems, we tackled the problem in this paper from a FBSDE approach, as also widely used in the context of optimal control problems \cite{Carmona}. Our approach sheds new light on the underlying structure of minimum variance estimation by carefully selecting the set of FBSDEs. In particular, we have strictly followed the classical FBSDE framework by only allowing for Brownian noise in the FBSDE formulations. We have also demonstrated that the two estimation formulations for $\sigma_T[f]$ lead, in fact, to a deterministic control law. This fact had been previously utilized in \cite{Kim22} in the context of scenario (ii) for deriving Zakai's equation; but its full implications only emerge in the context minimum variance estimation, its extension to the estimation of $\pi_T[f]$, and applications to stochastic optimal control. A further application could include the study of filter stability using (\ref{eq:Feynman-Kac}). 

It will also be of interest to explore numerical approximations of the proposed estimators and control laws based, for example, on \cite{CKSY21,EHJ17,EHJ21} and their relation to mean-field filtering equations \cite{meyn13,CRS22,taghvaei2023survey}. Finally, the semi-group approach to partially observed optimal control problems, as discussed in \cite{Nisio15}, allows one to avoid the BSPDE formulations (\ref{eq:BSDE_OC}) and (\ref{eq:BSPDE_K}), respectively. The arising (deterministic) HJB equation on a Hilbert space opens the door to alternative numerical approaches worth investigating. See, for example, \cite{PM13}.



\medskip

\noindent
{\bf Acknowledgment.} This work has been funded by Deutsche Forschungsgemeinschaft (DFG) - Project-ID 318763901 - SFB1294.
The authors thank Prashant Mehta for sharing his insight into the subject of this paper. The authors would also like to thank the Isaac Newton Institute for Mathematical Sciences, Cambridge, for support and hospitality during the programme  {\it The Mathematical and Statistical Foundation of Future Data-Driven Engineering} where work on this paper was undertaken. This work was supported by EPSRC grant no EP/R014604/1.

\bibliographystyle{plainurl}
%
\bibliography{bib-database}
%

\end{document}